\documentclass[12pt,leqno,a4paper]{article}

\usepackage{amssymb,amsmath,amsthm}
\usepackage[T1]{fontenc}

\title{\textbf{Square-Zero Basis}\\
\textbf{of Matrix Lie Algebras}}
\author{\textsc{R. Dur\'an D\'{\i}az$^1$,
V. Gayoso Mart\'{\i}nez$^2$}
\and \textsc{L. Hern\'andez Encinas$^2$,
and J. Mu\~{n}oz Masqu\'e$^2$}
\medskip \\
$^1$ Departamento de Autom\'atica,
Universidad de Alcal\'a,\\
E-28871 Alcal\'a
de Henares, Spain\\
\textit{E-mail:\/} \texttt{raul.duran@uah.es}\\
$^2$ Instituto de Tecnolog\'{\i}as F\'{\i}sicas
y de la Informaci\'on (ITEFI)\\Consejo Superior
de Investigaciones Cient\'{\i}ficas (CSIC),\\
E-28006 Madrid, Spain\\\textit{E-mail:\/}
\texttt{\{victor.gayoso,\!\! luis,\!\! jaime\}@iec.csic.es}
}
\date{}

\newtheorem{theorem}{Theorem}[section]
\newtheorem{proposition}[theorem]{Proposition}
\newtheorem{lemma}[theorem]{Lemma}
\newtheorem{corollary}[theorem]{Corollary}

\theoremstyle{remark}
\newtheorem{remark}[theorem]{Remark}
\newtheorem{notation}[theorem]{Notation}
\newtheorem{definition}[theorem]{Definition}
\newtheorem{example}[theorem]{Example}

\begin{document}

\maketitle

\begin{abstract}
\noindent A method is obtained to compute the maximum number
of functionally independent invariant functions under the action
of a linear algebraic group as long as its Lie algebra admits
a base of square-zero matrices. Some applications are also given.

\end{abstract}

\smallskip


\smallskip


\section{A class of Lie algebras}
Square-zero matrices have been dealt with in several settings
and with different purposes; for example, see \cite{Botha1},
\cite{Botha2}, \cite{Botha3}, \cite{KOV}, \cite{Mohammadian},
\cite{Takahashi}, among other authors. Below, we consider
such matrices in connection with the following question:

\begin{itemize}
\item[(*)]
Let $\mathbb{F}$ be a field. We try to find out whether a given
Lie subalgebra $\mathfrak{g}$ in $\mathfrak{gl}(n,\mathbb{F})$
admits a basis $\mathcal{B}$ (as a vector space over $\mathbb{F}$)
such that the square of any matrix in $\mathcal{B}$ is zero.
\end{itemize}
\begin{lemma}
\label{lemma1}Let $G\subseteq GL(n,\mathbb{F})$ be a linear algebraic
group with associated Lie algebra $\mathfrak{g}$. If $U$ is a square-zero
matrix in the Lie subalgebra
$\mathfrak{g}\subset \mathfrak{gl}(n,\mathbb{F})$, then $I+tU$
belongs to $G$, $\forall t\in\mathbb{F}$, where $I\in GL(V)$
denotes the identity map.
\end{lemma}
\begin{proof}
If $U\in \mathfrak{g}$ is a square-zero matrix, then
$H=\{I+tU:t\in \mathbb{F}\}$ is a linear algebraic group
of dimension $1$ with Lie algebra
$\mathfrak{h}=\{ tU:t\in\mathbb{F}\} $, and by virtue
of the assumption, we have
$\mathfrak{g}\cap \mathfrak{h}=\mathfrak{h}$.
Hence $\dim(G\cap H)=\dim H=1$,
so that $H=G\cap H$, or equivalently $H\subseteq G$.
\end{proof}
\begin{notation}
Notations and elementary properties of algebraic sets
and groups have been taken from Fogarty's book \cite{Fogarty}.

The Lie algebra $\mathfrak{g}$ of a linear algebraic group $G$
is identified to the Lie algebra of left-invariants derivations
(cf.\ \cite[3.17]{Fogarty}), namely
$\mathfrak{g}=\operatorname*{Der}
\nolimits_{\mathbb{F}}(\mathbb{F}[G])^G$.

If $A=(a_{ij})_{i,j=1}^n\in \mathfrak{gl}(n,\mathbb{F})$,
then the corresponding invariant derivation is given by
$D_A=\sum _{i,j,k=1}^n a_{ik}x_{ji}\frac{\partial}{\partial x_{jk}}$.
\end{notation}
\begin{definition}
Let $G\subseteq GL(n,\mathbb{F})$ be a linear algebraic group
and let $V=\mathbb{F}^n$. A function
$\mathcal{I}\in \mathbb{F[}V^\ast ]=S^\bullet (V^\ast )$ is said
to be $G$-invariant if $\mathcal{I}\left( g\cdot v\right)
=\mathcal{I}(v)$ for all $g\in G$ and all $v\in V$.
\end{definition}
The importance of (*) lies in the following result:
\begin{theorem}
\label{th1}
Let $G$ be a linear algebraic group, let
$\rho\colon G\to GL(n,\mathbb{F})$ be a linear representation
of $G$, and let
$\rho_\ast \colon \mathfrak{g}\to\mathfrak{gl}(n,\mathbb{F})$
be the homomorphism of Lie algebras induced by $\rho$.
If $V=\mathbb{F}^n$ and $\mathfrak{g}$ satisfies the property
\emph{(*)}, then every $G$-invariant function
$\mathcal{I}\in\mathbb{F[}V^\ast ]$ is a common first-integral
of the system of derivations $\rho_{\ast}(X)$,
$\forall X\in\mathfrak{g}$. Hence, the number of algebraically
independent $G$-invariant functions in $\mathbb{F[}V^\ast ]$
is upper bounded by the difference $n^2-r$, where $r$ is
the generic rank of the $\mathbb{F[}V^\ast ]$-module $\mathcal{M}$
spanned by all the derivations $\rho_\ast (X)$,
$\forall X\in \mathfrak{g}$; i.e., $r$ is the dimension
of the $\mathbb{F}(V^\ast)$-vector space
$\mathbb{F}(V^\ast )\otimes_{\mathbb{F}[V^\ast ]}\mathcal{M}$.
\end{theorem}
\begin{proof}
Let $\mathcal{B}$ be basis for $\mathfrak{g}$ by fulfilling
the condition in (*). By virtue of Lemma \ref{lemma1},
the matrix $I+tB$ belongs to $G$ and we have
$\mathcal{I}\left( (I+tB)\cdot v\right) =\mathcal{I}(v)$, for all
$t\in \mathbb{F}$, $B\in\mathcal{B}$, and by taking derivatives
at $t=0$, we deduce that $\rho _\ast (X)(\mathcal{I})=0$,
$\forall X\in \mathfrak{g}$, because the map
$\mathfrak{g\ni}X\mapsto\rho _\ast(X)
\in \operatorname*{Der}\mathbb{F}[V^\ast ]$ is
$\mathbb{F}$-linear. Consequently, if
$\mathcal{B}=\{ B_1,\dotsc,B_m\} $, then the
$\mathbb{F}[V^\ast ]$-module $\mathcal{M}$
is spanned by the invariant vector fields
$\rho _\ast (B_i)$, $1\leq i\leq m$, and the differential
$d\mathcal{I}\in \Omega _{\mathbb{F}}(\mathbb{F}[V^\ast ])$
of every invariant function $\mathcal{I}$ verifies
$d\mathcal{I}(X)=0$, $\forall X\in \mathcal{M}$,
and we can conclude.
\end{proof}
In classical invariant theory over complex numbers a method
for computing the maximum number of algebraically independent
invariants consists in solving the linear equations arising
from the system of first integrals of vector fields
$\rho _\ast (X)$, $X\in \mathfrak{g}$; e.g., see
\cite[Theorem 4.5.2]{Sturmfels}. Theorem \ref{th1} extends
this procedure to a class of linear representations in positive
characteristic.

It would also be interesting to adapt the algorithms given
in \cite{DK} to the linear representations of a linear
algebraic group whose Lie algebra satisfies
the condition (*) in positive characteristic.
\begin{remark}
\label{remark1}
As $\rho _\ast \colon\mathfrak{g}\to \mathfrak{gl}(n,\mathbb{F})$
is a homomorphism of Lie algebras, $\mathcal{M}$ is an
involutive submodule in $\operatorname*{Der}\mathbb{F}[V^\ast ]$.
In the real or complex cases, Frobenius theorem implies
that the maximum number of algebraically independent
first-integral functions of $\mathcal{M}$ is $n^2-r$ exactly,
but in general the upper bound $n^2-r$ is not necessarily
reached as several of these first-integral functions may be
fractional or even transcendental functions. Nevertheless, we have
\end{remark}
\begin{corollary}
\label{corollary1}
If $\rho\colon G\to GL(n,\mathbb{F})$ is as in \emph{Theorem \ref{th1}},
$\mathbb{F}$ is algebraically closed of characteristic zero, and
$G=SL(n,\mathbb{F})$ or $G=Sp(2n,\mathbb{F})$, then
the algebra $\mathbb{F}[V^\ast ]^G$ of $G$-invariant functions is an
$\mathbb{F}$-algebra of polynomials in $n^2-r$ variables.
\end{corollary}

\begin{proof}
According to \cite[Th\'{e}or\`{e}me 1]{KPV}, in the two cases of the statement
above we have $\mathbb{F}[V^\ast ]^G=\mathbb{F}[p_1,\dotsc,p_m]$, where
the polynomials $p_1,\dotsc,p_m$ are algebraically independent. Hence
their differentials $dp_1,\dotsc,dp_m$ form a basis of the dual module to
$\operatorname*{Der}\nolimits_{\mathbb{F}}(\mathbb{F}[V^\ast ])^G$ by
virtue of \cite[VIII. Proposition 5.5]{Lang}, and we thus obtain $m=n^2-r$.
\end{proof}
\begin{example}
\label{example1}
Let $GL(2,\mathbb{F})$ act on $V=\mathbb{F}^2\oplus S^2(\mathbb{F}^2)$
naturally and let $(v_1,v_2)$ be the standard basis for $V$; by setting
\[
\begin{array}
[c]{rl}
v= & xv_1+yv_2\in\mathbb{F}^2,\\
s= & z(v_1\otimes v_1)+t(v_{1}\otimes v_2+v_2\otimes v_1)
+u(v_2\otimes v_2)\in S^2(\mathbb{F}^2),
\end{array}
\]
we deduce that the basic invariant is the function
$\mathcal{I}_{1}\colon O\to\mathbb{F}$ defined on the Zariski
open subset $O$ of non-degenerate metrics as follows:
$\mathcal{I}_1(v,s)=s^\natural (v,v)$, where
$s^\natural \in S^2(\mathbb{F}^2)^\ast $ is the covariant symmetric
tensor induced by $s$, assuming $s$ is non-singular. In coordinates,
$\mathcal{I}_1(v,s)=\frac{2xyt-x^2u-y^2z}{t^2-zu}$. Hence
$\mathbb{F}[V^\ast ]^{GL(2,\mathbb{F})}=\mathbb{F}$ and
$\mathbb{F}(V^\ast )^{GL(2,\mathbb{F})}=\mathbb{F}(\mathcal{I}_1)$.
Nevertheless, the result depends strongly on the linear representation
being considered. For example, if we consider the natural representation
of $GL(2,\mathbb{F})$ on $V=\mathbb{F}^2\oplus S^2(\mathbb{F}^2)^\ast $,
then the basic invariant is the function
$\mathcal{I}_1^\prime (v,s^\ast )=s^\ast (v,v)$,
which is globally defined, and, in this case, we have
$\mathbb{F}[V^\ast ]^{GL(2,\mathbb{F})}=\mathbb{F}[\mathcal{I}_1^\prime ]$.
\end{example}
\begin{example}
\label{example2}If the natural representation of $SL(2,\mathbb{F})$ on
$V=\mathbb{F}^2\oplus S^2(\mathbb{F}^2)$ is considered, then, besides
$\mathcal{I}_{1}$, there exists another globally defined invariant, namely the
discriminant function, i.e., $\mathcal{I}_{2}(v,s)=zu-t^2$. Hence
$\mathcal{I}_{1}\mathcal{I}_{2}$ is also globally defined and we have
$\mathbb{F}[V^\ast ]^{SL(2,\mathbb{F})}
=\mathbb{F}[\mathcal{I}_1\mathcal{I}_2,\mathcal{I}_2]$.
\end{example}
\begin{example}
\label{example3}
A more complex example is the following: If $V$ is a six-dimensional
$\mathbb{F}$-vector space and $\Omega \colon V\times V\to \mathbb{F}$
is a non-degenerate alternate bilinear form, then, as
a computation shows, the generic rank of $\mathcal{M}$ for the linear
representation of $Sp(\Omega )$ on $\wedge ^3V^\ast $ is $18$; see
\cite{MP} for the details. As $\dim \wedge ^3V^\ast =20$, it follows that
there exist two invariant functions, both of them polynomial functions.
\end{example}
\begin{example}
\label{example4}
Given $A\in\mathfrak{gl}(2,\mathbb{C})\diagdown\{0\}$, let $X$
be the infinitesimal generator of the one-parameter group $\exp(tA)$,
$t\in\mathbb{C}$. Let $\alpha,\beta$ be the eigenvalues of $A$. We distinguish
several cases. If $\alpha\neq\beta$, $\alpha\beta\neq0$, then the vector field
$X$ admits a first integral in $\mathbb{C}(x,y)$ if and only if
$\alpha ^{-1}\beta \in \mathbb{Q}$; otherwise, every non-constant first integral
of $X$ is a transcendental function. If $\alpha\beta=0$, then $X$ admits a first
integral in $\mathbb{C}[x,y]$. If $\alpha =\beta \neq 0$ and the annihilator
polynomial of $A$ is $(\lambda -\alpha)^2$, then
$X=\alpha x\frac{\partial }{\partial x}+(1+\alpha y)\frac{\partial}{\partial y}$
and its basic first integral is the function
$\mathcal{I}=x\exp \left( -\alpha y/x\right) $. If the annihilator is
$\lambda -\alpha $, then $X$ admits a first integral in
$\mathbb{C}(x,y)$. Finally, If $\alpha =\beta =0$ then the annihilator $A$ is
$\lambda ^2$ and $X$ admits the function $x$ as a first integral.
\end{example}
\section{The property (*) studied}
\begin{notation}
\label{notation1}
Let $(v_i)_{i=1}^n$ be the standard basis for $\mathbb{F}^n$
with dual basis $(v^i)_{i=1}^n$. Every matrix
$A\in \mathfrak{gl}(n,\mathbb{F})$ is identified with the endomorphism
on $\mathbb{F}^n$ to which such matrix corresponds in the basis
$(v_1,\dotsc,v_n)$. If $x=x^hv_h$, then $E_{ij}(x)=x^jv_i$, or
equivalently $E_{ij}(v_k)=\delta _{jk}v_i$, which means that $E_{ij}$ is
the matrix with $1$ in the entry $(i,j)$ and $0$ in the rest of entries.
Therefore $(E_{hi}\circ E_{jk})(v_{l})=\delta _{kl}\delta_{ij}v_{h}$.
Hence
\begin{equation}
\begin{array}
[c]{rl}
E_{hi}\circ E_{jk}= & \delta_{ij}E_{hk},\\
(E_{hi})^2=\delta_{hi}E_{hi}= & \left\{
\begin{array}
[c]{rl}
0, & i\neq h\\
E_{hh}, & i=h
\end{array}
\right.
\end{array}
\label{EE}
\end{equation}
The Lie algebra of $n\times n$ traceless matrices with entries
in $\mathbb{F}$ is denoted by $\mathfrak{sl}(n,\mathbb{F})$.
The Lie algebra of $n\times n$ skew-symmetric matrices with entries
in $\mathbb{F}$ is denoted by $\mathfrak{so}(n,\mathbb{F})$.
The Lie algebra of $2n\times2n$ matrices $X$ with entries
in $\mathbb{F}$ such that $X^TJ_n+J_nX=0$, where
\[
J_n=\left(
\begin{array}
[c]{cc}
0 & I_n\\
-I_n & 0
\end{array}
\right) ,
\]
and $I_n\in \mathfrak{gl}(n,\mathbb{F})$ is the identity
matrix, is denoted by $\mathfrak{sp}(2n,\mathbb{F})$.
\end{notation}
By using the formulas \eqref{EE} and the standard basis for the Lie algebra
$\mathfrak{sl}(n,\mathbb{F})$, i.e., the $n^2-1$ matrices $E_{hi}$,
$h\neq i$, $h,i=1,\dotsc,n$; $E_{hh}-E_{11}$, $2\leq h\leq n$, we obtain
\begin{proposition}
\label{prop1}
The matrices
\[
\begin{array}
[c]{rll}
E_{hi}, & h\neq i, & h,i=1,\dotsc,n,\\
E_{hh}-E_{11}-E_{1h}+E_{h1}, & 2\leq h\leq n, &
\end{array}
\]
are a basis for $\mathfrak{sl}(n,\mathbb{F})$ fulfilling the property
\emph{(*)}.
\end{proposition}
\begin{proposition}
\label{prop1.5}
If the characteristic of $\mathbb{F}$ is either zero or is
positive $p$ and $p$ does not divide $n$, then the identity
matrix $I\in \mathfrak{gl}(n,\mathbb{F})$ cannot be written
as a sum of square-zero matrices.
\end{proposition}
\begin{proof}
If $I=N_1+\ldots +N_k$, $(N_i)^2=0$, $1\leq i\leq k$, as the trace
of a nilpotent matrix vanishes, then by taking traces on both sides
in the previous equation we have $n=0$ if the characteristic of
$\mathbb{F}$ is zero, and $n\equiv 0\bmod p$ if the characteristic is $p$.
\end{proof}
\begin{corollary}
\label{corollary1.5}
If the characteristic of $\mathbb{F}$ is $2$, then
the identity matrix
$I\in \mathfrak{gl}(n,\mathbb{F})$ can be written
as a sum of square-zero matrices if and only if
$n$ is even.
\end{corollary}
\begin{proof}
If $n$ is odd the result follows from Proposition \ref{prop1.5}.
If $n=2m$, then let $(v_i)_{i=1}^n$ be a basis for $V=\mathbb{F}^n$
with dual basis $(v_i^\ast )_{i=1}^n$. The space
$\operatorname*{End}(V)$ is identified to $V^\ast \otimes V$ as usual, so
that we have
\[
I=\sum\nolimits_{i=1}^m
\left( v_{2i-1}^\ast \otimes v_{2i-1}+v_{2i}^\ast \otimes v_{2i}\right) .
\]
Thus
\[
\begin{array}
[c]{rl}
\sum\nolimits_{i=1}^m\left( v_{2i-1}^\ast +v_{2i}^\ast \right)
\otimes\left( v_{2i-1}+v_{2i}\right) = & I+\sum\nolimits_{i=1}^m
\left(
v_{2i-1}^\ast \otimes v_{2i}+v_{2i}^\ast \otimes v_{2i-1}
\right) ,
\end{array}
\]
$A_i=\left( v_{2i-1}^\ast +v_{2i}^\ast \right) \otimes
\left( v_{2i-1}+v_{2i}\right) $, $v_{2i-1}^\ast \otimes v_{2i}$,
and $v_{2i}^\ast \otimes v_{2i-1}$, $1\leq i\leq m$,
are square-zero matrices, and for every $1\leq h\leq n$, we have
\[
\begin{array}
[c]{rl}
A_i(v_h)= & \left( \delta _{h,2i-1}
+\delta _{h,2i}\right) \left( v_{2i-1}+v_{2i}\right) ,\\
\left(  A_{i}\right) ^2(v_h)= & \left( \delta_{h,2i-1}
+\delta _{h,2i}\right) A_i\left( v_{2i-1}+v_{2i}\right) \\
= & 2\left( \delta_{h,2i-1}+\delta_{h,2i}\right)
\left( v_{2i-1}+v_{2i}\right)\\
= & 0\operatorname{mod}2,\\
\left( v_{2i-1}^\ast \otimes v_{2i}\right) ^2v_h= &
\delta _{h,2i-1}\left( v_{2i-1}^\ast \otimes v_{2i}\right)
\left( v_{2i}\right)
\\
= & 0,\\
\left( v_{2i}^\ast \otimes v_{2i-1}\right)  ^2v_h
= & \delta _{h,2i}\left( v_{2i}^\ast
\otimes v_{2i-1}\right) \left( v_{2i-1}\right) ,\\
= & 0.\\
\end{array}
\]
\end{proof}
Similarly, by starting with the standard basis for the symplectic
Lie algebra $\mathfrak{sp}(2n,\mathbb{F})$, i.e.,
\[
\begin{array}
[c]{llll}
E_{i,n+i}, & E_{n+i,i}, & E_{ii}-E_{n+i,n+i}, & 1\leq i\leq n,\\
E_{ij}-E_{n+i,n+j}, & E_{ji}-E_{n+j,n+i}, & E_{i,n+j}-E_{j,n+i},
& 1\leq i<j\leq n,\\
&  & E_{n+i,j}-E_{j+n,i}, & 1\leq i<j\leq n,
\end{array}
\]
we obtain
\begin{proposition}
\label{prop2}The matrices
\[
\begin{array}
[c]{lll}
E_{i,n+i,},E_{n+i,i}, & 1\leq i\leq n, & \\
E_{ii}-E_{n+i,n+i}+E_{i,n+i}-E_{n+i,i}, & 1\leq i\leq n, & \\
E_{ij}-E_{n+i,n+j}, & E_{ji}-E_{n+j,n+i}, & 1\leq i<j\leq n,\\
E_{i,n+j}-E_{j,n+i}, & E_{n+i,j}-E_{j+n,i}, & 1\leq i<j\leq n,
\end{array}
\]
are a basis for $\mathfrak{sp}(2n,\mathbb{F})$ fulfilling the property
\emph{(*)}.
\end{proposition}
Similarly, we have
\begin{proposition}
\label{prop2_bis}
The standard basis $E_{hi}$, $1\leq h< i\leq n$,
of the Lie subalgebra of strictly upper triangular matrices in
$\mathfrak{gl}(n,\mathbb{F})$ satisfies the property \emph{(*)}.
\end{proposition}

As for the Lie algebra $\mathfrak{so}(n,\mathbb{F})$, with basis
$E_{hi}-E_{ih}$, $1\leq h<i\leq n$, we have
\begin{proposition}
\label{prop3}
Let $x_1,\dotsc,x_n$ be the column vectors of a matrix
$X$ in $\mathfrak{so}(n,\mathbb{F})$ of rank $r$, and let
$x_{i_1},\dotsc,x_{i_r}$, $1\leq i_1<\ldots <i_r\leq n$, be $r$ linearly
independent column vectors of $X$. The necessary and sufficient condition
for the square of $X$ to be zero is that the subspace
$\left\langle x_{i_1},\dotsc,x_{i_r}\right\rangle $ is totally isotropic
with respect to the scalar product $\left\langle \cdot,\cdot\right\rangle $
given by $\left\langle v_i,v_j\right\rangle =\delta _{ij}$, $i,j=1,\dotsc,n$.
\end{proposition}
\begin{proof}
As $X$ is skew-symmetric, for all $i,j=1,\dotsc,n$, we have
\[
\left( X^2\right) _{ij}=\sum\nolimits_{h=1}^nx_{ih}x_{hj}
=-\sum \nolimits_{h=1}^nx_{hi}x_{hj}
=-\left\langle x_{i},x_{j}\right\rangle .
\]
Hence $X^2=0$ if and only if $\left\langle x_i,x_j\right\rangle =0$
for $1\leq i\leq j\leq n$.

Moreover, if $k,l\notin \{ i_1,\dotsc,i_r\} $, then
$x_k=\sum _{a=1}^r\lambda _{ka}x_{i_a}$,
$x_l=\sum _{b=1}^r\lambda _{lb}x_{i_b}$; consequently
$\left\langle x_k,x_l\right\rangle
=\sum\nolimits_{a=1}^r\sum\nolimits_{b=1}^r\lambda _{ka}\lambda _{lb}
\left\langle x_{i_a},x_{i_b}\right\rangle $. It follows that $X^2=0$
if and only if $\left\langle x_{i_a},x_{i_b}\right\rangle =0$
for $1\leq a\leq b\leq r$.
\end{proof}
\begin{corollary}
\label{corollary2}
If the ground field $\mathbb{F}$ is formally real, then the
only matrix $X$ in $\mathfrak{so}(n,\mathbb{F})$ with $X^2=0$
is the zero matrix.
\end{corollary}
\begin{proof}
In fact, if $x=\sum_{i=1}^nx^iv_i$, then: $\left\langle x,x\right\rangle
=\sum _{i=1}^n(x^i)^2$ and, by virtue of the hypothesis, it follows that
the only totally isotropic subspace for
$\left\langle \cdot,\cdot \right\rangle $ is $\{ 0\} $.
\end{proof}
\begin{remark}
\label{remark3}
If the characteristic of $\mathbb{F}$ is $\neq2$, then the only
matrix $X$ in $\mathfrak{so}(2,\mathbb{F})$ such that $X^2=0$ is the zero
matrix, as $\left[  \alpha(E_{12}-E_{21})\right]  ^2=-\alpha ^2I$. The same
holds for $\mathfrak{so}(3,\mathbb{F})$. In fact, if
\[
X=a(E_{12}-E_{21})+b(E_{13}-E_{31})+c(E_{23}-E_{32}),
\]
then, as the matrix $X^2$ is symmetric, the condition $X^2=0$
leads one to the following system of six equations: $a^2+b^2=0$,
$ab=0$; $a^2+c^2=0$, $ac=0$; $b^2+c^2=0$, $bc=0$. Hence $a+b=0$,
$a+c=0$, $b+c=0$, and consequently $a=b=c=0$.
\end{remark}
Next we study the property (*) for certain Lie algebras in characteristic $2$
following the notations and results of \cite{AIS}. Assume the characteristic
of $\mathbb{F}$ is $2$, let $f\colon V\times V\to\mathbb{F}$ be a
bilinear form and let $L(f)\subseteq \mathfrak{gl}(V)$ be its associated Lie
subalgebra, i.e., $L(f)=\{X\in \mathfrak{gl}(V):f(X(u),v))=f(u,X(v)),
\forall u,v\in V\}$. (Recall that we are in characteristic $2$.) If in addition
$\mathbb{F}$ is algebraically closed, then according to
\cite[Theorem 1.1]{AIS} the Lie algebra $L(f)$ is reductive if and only if
either [i] $f=0$ and $n\neq 2$, in which case $L(f)=\mathfrak{gl}(V)$,
[ii] or $n=2m+1$ and $f$ admits a Gram matrix $J_{2m+1}$, in which case
$L(f)$ is Abelian of dimension $m+1$, [iii] or else $f$ admits as Gram matrix
a direct sum of matrices of the types indicated below, in which case
$L(f)$ is isomorphic to the direct sum of the Lie algebras associated to
these matrix summands:

\textsc{Type} $0$:

$A=\left(
\begin{array}
[c]{cc}
0 & J_{m}\\
I_{m} & 0
\end{array}
\right) $, $L(A)$ Abelian of dimension $m$;

$B=\left(
\begin{array}
[c]{cc}
0 & 0\\
I_{m} & 0
\end{array}
\right) $, $L(B)\cong\mathfrak{gl}(m)$, $m>2$.

\textsc{Type} $\lambda$, $\lambda \in \mathbb{F}$, $\lambda\neq 1$:

$A=\left(
\begin{array}
[c]{cc}
0 & J_{m}(\lambda)\\
I_{m} & 0
\end{array}
\right) $, $L(A)$ Abelian of dimension $m$;

$B=\left(
\begin{array}
[c]{cc}
0 & \lambda I_m\\
I_{m} & 0
\end{array}
\right)  $, $L(B)\cong \mathfrak{gl}(m)$, $m>2$.

\textsc{Type} $1$:

$A=\Gamma _m$, $m$ odd, $L(A)$ Abelian of dimension $\frac{1}{2}(m+1)$;

$B=\left(
\begin{array}
[c]{cc}
0 & J_{2}(1)\\
I_{2} & 0
\end{array}
\right) $, $L(B)$ Abelian of dimension $4$;

$C=I_m$, $m>2$, $L(C)\cong \mathfrak{so}(m)$;

$D=\left(
\begin{array}
[c]{cc}
0 & I_m\\
I_m & 0
\end{array}
\right) $, $m>2$, $L(D)\cong \mathfrak{sp}(2m)$.

In the case [i] the condition (*) does not hold for $L(f)$
as this condition never holds for $\mathfrak{gl}(m,\mathbb{F})$.

In the case [ii] the condition (*) does not hold for
$L(f)$ whatever the odd integer $n>1$.

In the case of the matrix $A$ in \textsc{Type} $0$,
$L(A)$ is the Abelian Lie algebra generated by the powers
\[
\begin{array}
[c]{ll}
\left(
\begin{array}
[c]{cc}
(J_m)^T & 0\\
0 & J_m
\end{array}
\right) ^i, & 0\leq i\leq m-1.
\end{array}
\]
From Proposition \ref{prop2_bis} it follows that the property (*) holds true
for the algebra $L(A)$; but the property (*) does not hold for $L(B)$ in
\textsc{Type} $0$.
The Lie algebras $L(A)$ and $L(B)$ in \textsc{Type} $1$ do
not verify the property (*). As for $L(C)\cong \mathfrak{so}(m)$ in
\textsc{Type} $1$, the property (*) depends on the nature of the ground field,
as we have seen above and the Lie algebra $L(D)$ does verify the property (*)
as follows directly from Proposition \ref{prop2}.

\medskip

\noindent \textit{Acknowledgments:\/}
This research has been partially supported by Ministerio de Economía,
Industria y Competitividad (MINECO), Agencia Estatal de Investigación
(AEI), and Fondo Europeo de Desarrollo Regional (FEDER, UE) under
project COPCIS, reference TIN2017-84844-C2-1-R, and by Comunidad de
Madrid (Spain) under project reference S2013/ICE-3095-CIBERDINE-CM,
also co-funded by European Union FEDER funds.

\end{document}